\pdfoutput=1 
\documentclass[10pt]{amsart}
\usepackage[T2A,T1]{fontenc}
\usepackage[utf8]{inputenc}
\usepackage[russian,english]{babel}

\newcommand{\TT}{\mathbb{T}}
\newcommand{\RR}{\mathbb{R}}

\newcommand{\ZZ}{\mathbb{Z}}
\newcommand{\lp}{\left}
\newcommand{\rp}{\right}

\usepackage{amsmath,amsthm,amssymb}
\usepackage[linktoc=all]{hyperref}

\theoremstyle{plain}
\newtheorem{theorem}{Theorem}
\newtheorem{lemma}{Lemma}
\newtheorem{corollary}[theorem]{Corollary}
\newtheorem{question}{Question}

\newtheorem*{claim*}{Claim}

\theoremstyle{definition}

\newtheorem*{definition*}{Definition}

\theoremstyle{remark}

\theoremstyle{definition}
\newtheorem{example}{Example}

\newcommand{\PP}{\mathbb{P}}
\newcommand{\EE}{\mathbb{E}}

\title{Poissonian correlation of higher order differences}
\address{Yale University, 10 Hillhouse Ave, New Haven, CT 06511, USA}
\email{alex.cohen@yale.edu}
\author{Alex Cohen}
\subjclass[2010]{28E99, 42A82}
\keywords{Pair correlation, difference sequences.}

\date{November 2020}

\begin{document}

\begin{abstract} 
A sequence $(x_n)_{n=1}^{\infty}$ on the torus $\mathbb{T}$ exhibits Poissonian pair correlation if for all $s>0$,
\begin{equation*}
     \lim_{N\to\infty} \frac{1}{N}\#\left\{1\leq m\neq n \leq N : |x_m-x_n| \leq \frac{s}{N}\right\} = 2s.
 \end{equation*} 
 It is known that this condition implies equidistribution of $(x_n)$. We generalize this result to \textit{four-fold differences}: if for all $s > 0$ we have
 \begin{equation*}
     \lim_{N\to\infty} \frac{1}{N^2}\#\left\{\substack{1\leq m,n,k,l\leq N\\\{m,n\}\neq\{k,l\}} : |x_m+x_n-x_k-x_l| \leq \frac{s}{N^2}\right\} = 2s
 \end{equation*}
 then $(x_n)_{n=1}^{\infty}$ is equidistributed. This notion generalizes to higher orders, and for any $k$ we show that a sequence exhibiting \textit{$2k$-fold Poissonian correlation} is equidistributed. In the course of this investigation we obtain a discrepancy bound for a sequence in terms of its closeness to $2k$-fold Poissonian correlation. This result refines earlier bounds of Grepstad \& Larcher and Steinerberger in the case of pair correlation, and resolves an open question of Steinerberger. 
\end{abstract}
\maketitle

\section{Introduction}
Let $(x_n)_{n=1}^{\infty}$ be a sequence on $\TT$. If each $x_n$ is drawn independently and uniformly at random from $\TT$, then for all $s> 0$
\begin{equation}\label{eq:pair_correlation}
    \lim_{N\to\infty} \frac{1}{N}\#\lp\{ 1\leq m\neq n\leq N\; :\; |x_m-x_n| \leq \frac{s}{N} \rp\} = 2s\quad \text{almost surely}.
\end{equation}
This property is called Poissonian pair correlation, and has been studied intensely in recent years; for example by Heath-Brown \cite{H}, Aistleitner, Larcher, \& Lewko \cite{ACL}, Steinerberger \cite{SSHighDim,SSDisc}, and Marklof \cite{M}. A sequence exhibits Poissonian pair correlation if the pairwise differences behave similarly at small scales to those of a Poissonian random variable. In general, studying the pairwise differences of a sequence is very difficult. For example, Weyl proved that for $\alpha$ an irrational number, the sequence $(\{\alpha n^2\})_{n=1}^{\infty}$ is equidistributed on $\TT$. The pair correlation of such sequences is related to the spacings between eigenvalues of certain Hamiltonian matrices; motivated by this connection to physics, mathematicians asked if sequences of the form $(\{\alpha n^d\})_{n=1}^{\infty}$ exhibit Poissonian pair correlation. This problem inspired much work, for instance, by Boca \& Zaharescu \cite{BZ}, El-Baz, Marklof \& Vinogradov \cite{EMV}, Heath-Brown \cite{H}, Nair \& Pollicott \cite{NP}, Rudnick, Sarnak, \& Zaharescu \cite{RSZ}, Aistleitner, Larcher, \& Lewko \cite{ACL}, Walker \cite{W}, Steinerberger \cite{SSHighDim,SSDisc}, and Marklof \cite{M} among others. Despite intense study of pairwise differences for sequences on $\TT$, it was only proven very recently that Poissonian pair correlation implies equidistribution. Aistleitner, Lachmann \& Pausinger \cite{ALP} and Grepstad \& Larcher \cite{GL} proved this result independently in 2017.  
\begin{theorem}[Aistleitner, Lachmann \& Pausinger, Grepstad \& Larcher] \label{poissonian_pair}
If a sequence $(x_n)$ exhibits Poissonian pair correlation, it is equidistributed.
\end{theorem}

Theorem \ref{poissonian_pair} can be interpreted as a difference theorem at scale: if the differences $x_n - x_m$ of a sequence behave like a Poissonian random variable at small scales, then the sequence itself is equidistributed. A global version of this result has been known for some time; Weyl \cite{Weyl} proved that the sequence of differences $(x_n - x_m)$ is equidistributed on $\TT$ if and only if the sequence $(x_n)$ is equidistributed. This result comes with explicit discrepancy bounds.
\begin{theorem}[Cassels \cite{C}, 1953] \label{difference_discrepancy}
Let $x_1,\ldots,x_N$ be a finite sequence and $(y_n)_{n=1}^{N^2}$ be the sequence consiting of pairwise differences $(x_m-x_n)_{1\leq m\neq n\leq N}$. Let $D$ be the discrepancy of $(x_1,\ldots,x_N)$ and $F$ the discrepancy of $(y_1,\ldots,y_{N^2})$. Then
\begin{equation*}
    D \leq c\sqrt{F}(1+|\log F|).
\end{equation*}
\end{theorem}
Earlier results of this kind were given by Vinogradoff \cite{V}, van der Corput \& Pisot \cite{VP}, and Koksma \cite{K}. Given a sequence $(x_n)$, we may apply Weyl's difference theorem twice. Taking the difference sequence of the difference sequence, we see that $(x_n)$ is equidistributed if and only if the sequence of four-fold differences $(x_m + x_n - x_k - x_l)$ is equidistributed. By analogy with Theorem \ref{poissonian_pair}, we show that if the four fold differences $x_n+x_m-x_k-x_l$ behave like a Poissonian random variable at small scales, then the sequence is equidistributed. We also generalize to higher order differences: if sums of the form $(x_{a_1}+\cdots+x_{a_k})-(x_{b_1}+\cdots+x_{b_k})$ behave like a Poissonian random variable at small scales, then the sequence is equidistributed. In the proof of this result we derive an estimate on the Weyl sums, which we use to obtain an explicit discrepancy bound on the original sequence. This bound turns out to be quite strong: it refines earlier results of Grepstad \& Larcher \cite{GL} and Steinerberger \cite{SSDisc}, and allows us to resolve an open problem of Steinerberger regarding pair correlation and discrepancy. 

\section{Results}
\subsection{Introduction}
We introduce a generalization of Poissonian pair correlation to higher order differences. 
\begin{definition*}\label{kfold-correlation}
A sequence $(x_n)_{n=1}^{\infty}$ has \textit{Poissonian correlation of $2k$-fold differences} if for all $s > 0$, 
\small
\begin{equation}\label{eq_2kfold_poissonian}
     \lim_{N\to\infty} \frac{1}{N^k}\#\lp\{\substack{1 \leq a_j,b_j \leq N\\\{a_j\}\neq \{b_j\}}\; :\; \lp|\sum_{j=1}^k x_{a_j} - \sum_{j=1}^k x_{b_j}\rp| \leq \frac{s}{N^k}\rp\} = 2s.
\end{equation}
\normalsize   
\end{definition*}
It is worth mentioning that the notion of $2k$-fold correlation is different from the notion of \textit{local $m$-level correlation}, which is common in the literature (see e.g \cite{RSZ}). The notion of local $m$-level correlation is a straightforward generalization of Poissonian pair correlation to several points; just like Poissonian pair correlation it only detects points very close to each other, and it takes place at scale $1/N$. On the other hand, $2k$-fold correlations are a global phenomenon, and take place at scale $1/N^k$. Our intuition is that higher order correlation is in some sense a weaker property than pair correlation, because taking higher order differences destroys more and more information from the original sequence. 

\subsection{Examples}
We provide examples showing that the notions of Poissonian pair correlation and $2k$-fold Poissonian correlation are distinct. See \S \ref{section-examples} for details and more examples. 

\subsubsection{Pair correlation does not imply four-fold correlation}
Consider the sequence $(x_1,-x_1,x_2,-x_2,\ldots)$ where each $x_n$ is chosen uniformly and independently at random on $\TT$. This sequence will have Poissonian pair correlation because local spacings behave randomly, but it will not have four-fold correlation, because there will be approximately $N^2$ sums of the form 
\begin{equation*}
    x_m + (-x_m) - x_n - (-x_n) = 0.
\end{equation*}
This example illustrates that Poissonian pair correlation is a local property, whereas four-fold correlation is a global property. See \S \ref{section-examples} Example \ref{pair_not_fourfold} for more details, and Example \ref{additive-energy-example} for a sequence of the form $(\{n_a\alpha\})_{a=1}^{\infty}$ that exhibits Poissonian pair correlation but not four-fold correlation. 

\subsubsection{Four-fold correlation does not imply Pair correlation}
Let
\begin{equation*}
    (y_n) = (x_1+p_1,x_1+p_2,x_2+p_3,x_2+p_4,\ldots)
\end{equation*}
be a sequence with each $x_n$ chosen uniformly and independently at random on $\TT$, and $p_n$ a perturbation chosen uniformly and independently at random in the range $[-n^{-3/2}, n^{-3/2}]$. 
The sequence $(y_n)$ does not have Poissonian pair correlation because $y_{2n} - y_{2n+1}$ is at scale $n^{-3/2}$, so for large enough $n$, $|y_{2n} - y_{2n+1}| < s/n$. 
On the other hand $(y_n)$ does have four-fold correlation because the perturbations $p_n$ are at scale $n^{-3/2}$, and the notion of four-fold correlation only detects perturbations at scale $n^{-2}$.
This example exploits the difference in scaling regimes between pair correlation and four-fold correlation. Later on we will introduce a variation of four-fold correlation that takes place at the same scaling regime as Poissonian pair correlation--see \S\ref{section-examples} Example \ref{fourfold_not_pair} for details, and Example \ref{fourfold_not_pair_rate1} for a comparison of the two concepts at the same scaling regime. 

\subsection{Poissonian Correlation of \texorpdfstring{$2k$}{2k}-fold Differences implies Equidistribution}
We show an analogue of Theorem \ref{poissonian_pair} for higher order differences.
\begin{theorem}\label{kway}
If a sequence $(x_n)$ exhibits Poissonian correlation of $2k$-fold differences, then $(x_n)$ is equidistributed.
\end{theorem}
In the case of $k = 1$, this reduces to Theorem \ref{poissonian_pair} which was proved independently by Aistleitner, Lachmann \& Pausinger \cite{ALP} and Grepstad \& Larcher \cite{GL}. Our proof generalizes Steinerberger's proof of Theorem \ref{poissonian_pair} in \cite{SSHighDim}; in the proof we derive the following estimate on the Weyl sums of the sequence, which is analogous to Steinerberger's estimate in the $k=1$ case. 
\begin{lemma}\label{lem:weyl-sum-estimate}
If a sequence $(x_n)$ exhibits Poissonian correlation of $2k$-fold differences, then
\begin{equation*}
    \limsup_{N\to\infty} \sum_{m=1}^{N/(4t)} \lp|\frac{1}{N}\sum_{j=1}^N e^{2\pi i m x_j}\rp|^{2k} \leq \frac{k!}{t}\quad \text{for all } t > 0.
\end{equation*}
\end{lemma}
This Lemma can be used to obtain discrepancy bounds for the original sequence. Notice that as $k$ increases, this condition becomes weaker: because we are taking a higher power, the Weyl sums can be larger. This fact agrees with our intuition that Pair correlation is a stronger condition than higher order correlation, because as we take higher order differences, we destroy more information. It is worth mentioning that although we do not analyze the higher dimensional case in this paper, an analogous version of Theorem \ref{kway} should go through for sequences in $\TT^m$ with the Euclidean norm. Indeed, our proof is based on Steinerberger's proof of Theorem \ref{poissonian_pair}, and Steinerberger's proof was developed to generalize Theorem \ref{poissonian_pair} to higher dimensions. 

\subsection{Weak Poissonian correlation.} In addition to Poissonian correlation of $2k$-fold differences, we consider a continuous family of similar conditions at various scaling regimes. For any $\alpha \geq 0$, we say that a sequence exhibits \textit{$2k$-fold Poissonian correlation at rate $\alpha$} if for all $s > 0$,
\begin{equation}\label{eq_poissonian_rate_alpha_kfold}
\lim_{N\to\infty} \frac{1}{N^{2k-\alpha}}\#\lp\{\substack{1 \leq a_j,b_j\leq N\\ \{a_j\}\neq \{b_j\}}\; :\; \lp|\sum_{j=1}^k x_{a_j} - \sum_{j=1}^k x_{b_j}\rp| \leq \frac{s}{N^{\alpha}}\rp\} = 2s.
\end{equation}
The interesting range of $\alpha$ to consider is $0 \leq \alpha \leq k$, as for $\alpha > k$ the scaling is too small. This notion was first introduced by Nair and Pollicott in the context of pair correlation \cite{NP}, and Steinerberger proved Poissonian pair correlation at rate $\alpha$ for any $0 < \alpha < 1$ is sufficient to conclude equidistribution \cite{SSQuant}. We generalize his result to higher order differences.

\begin{theorem}\label{kway_alpha}
If a sequence $(x_n)$ exhibits Poissonian correlation of $2k$-fold differences at rate $\alpha$ for some $0\leq \alpha \leq k$, then $(x_n)$ is equidistributed.
\end{theorem}

Notice that the $\alpha = 0$, $k=1$ case of this theorem reduces to Weyl's result that if a difference sequence $(x_n-x_m)$ is equidistributed, the original sequence $(x_n)$ is equidistributed as well. The $\alpha = k$ case recovers Theorem \ref{kway}. Choices $0 < \alpha < k$ interpolate between the global differences of Weyl's theorem and the critical scaling regime of Poissonian correlation; the case $\alpha=1$ may be of particular interest.

\begin{corollary}
Let $(x_n)$ be a sequence. If for some $k > 0$ we have
\begin{equation*}
\lim_{N\to\infty} \frac{1}{N^{2k-1}}\#\lp\{\substack{1 \leq a_j,b_j\leq N\\ \{a_j\}\neq \{b_j\}}\; :\; \lp|\sum_{j=1}^k x_{a_j} - \sum_{j=1}^k x_{b_j}\rp| \leq \frac{s}{N}\rp\} = 2s\quad \text{for all } s > 0,
\end{equation*}
then $(x_n)$ is equidistributed.
\end{corollary}
This corollary allows us to compare Poissonian pair correlation and $2k$-fold Poissonian correlation at the same scaling regime. In \S\ref{section-examples} Example \ref{fourfold_not_pair_rate1} we show that the sequence $(x_1,x_1,x_2,x_2,\ldots)$ with $x_n$ chosen uniformly and independently on $\TT$ has four-fold Poissonian correlation at rate $\alpha = 1$, but not Poissonian pair correlation.

\subsection{Discrepancy bounds from the exponential sum estimate.}
\subsubsection{Introduction}
In our proof of Theorem \ref{kway} we derive Lemma \ref{lem:weyl-sum-estimate}, yielding a strong estimate on the Weyl sums for the sequence. We then use the Erd\"os-Tur\'an inequality to obtain discrepancy bounds from this estimate. For $(x_n)$ a sequence on $\TT$, we let $\mu_N = \frac{1}{N}(\delta_{x_1}+\cdots+\delta_{x_N})$ be the normalized counting measure for the first $N$ points. The \textit{discrepancy} of the sequence is then
\begin{equation*}
    D_N = \sup_{\substack{A\subset \TT\\A\,\text{interval}}} \lp|\frac{1}{N}\#\{1 \leq n \leq N\, :\, x_n \in A\} - |A|\rp| = \sup_{\substack{A\subset \TT\\A\,\text{interval}}}|\mu_N(A) - |A||.
\end{equation*}
The classical Erd\"os-Tur\'an inequality \cite{ET1,ET2} says that for any probability measure $\mu$ on the circle $\TT$,
\begin{equation*}\label{erdos-turan}
    \sup_{\substack{A\subset \TT\\A\,\text{interval}}} |\mu(A) - |A|| \leq C \lp(\frac{1}{M} + \sum_{m=1}^M \frac{|\hat{\mu}(m)|}{m}\rp)\quad \text{for all } M>0
\end{equation*}
where $C$ is a universal constant. Taking $\mu = \mu_N$ above, we obtain
\begin{equation*}
    D_N \lesssim \frac{1}{M}+\sum_{m=1}^M \frac{1}{m}\lp|\frac{1}{N}\sum_{n=1}^N e^{2\pi i m x_n}\rp|\quad \text{for all } M > 0
\end{equation*}
where the notation $\lesssim$ suppresses the constant $C$. Using this estimate, we obtain
\begin{equation}\label{discrepancy_estimate_exp_sum}
    D_N \lesssim \frac{T}{N^K} + k\lp(\sum_{m=1}^{N^K/(4T)} \lp|\frac{1}{N}\sum_{n=1}^N e^{2\pi i m x_n}\rp|^{2k}\rp)^{\frac{1}{2k}}.
\end{equation} 
for all $T, N > 0$. We now have a bound for the discrepancy of a sequence in terms of the exponential sum estimate arising in Lemma \ref{lem:weyl-sum-estimate}. However, we have no a priori information on the quantitative rate at which these exponential sums converge for large $N$. The following definition describes the rate at which a sequence achieves Poissonian correlation, and allows us to obtain an explicit discrepancy bound from Equation (\ref{discrepancy_estimate_exp_sum}).
\subsubsection{Explicit discrepancy bounds}
\begin{definition*}
For $(x_n)$ a sequence on $\TT$, let the \textit{Poissonian pair discrepancy} be
\begin{equation*}
    D_{T,N} = \max_{s \in \{0,\ldots,T\}} \lp(\frac{1}{N}\#\{1 \leq m\neq n \leq k\; :\; |x_m - x_n| < s/N\} - 2s\rp)
\end{equation*}
and more generally, for any $k > 0$ we let the \textit{$2k$-fold correlation discrepancy} be
\begin{equation*}
    D^{(2k)}_{T,N} = \max_{s\in\{0,\ldots,T\}} \lp(\frac{1}{N^k}\#\lp\{\substack{1 \leq a_j,b_j\leq N\\\{a_j\}\neq\{b_j\}}\; :\; \lp|\sum_{j=1}^k x_{a_j} - \sum_{j=1}^k x_{b_j}\rp| < \frac{s}{N^k}\rp\} - 2s\rp).
\end{equation*}
\end{definition*}
Correlation discrepancy measures quantitatively how quickly a sequence achieves Poissonian correlation. Notice that instead of taking an absolute difference from $2s$, we just take a difference; because we allow $s = 0$ in the definition, we have $D_{T,N}^{(2k)} \geq 0$. We take a difference rather than an absolute difference because we only need to ensure quantities such as $\frac{1}{N}\#\{1 \leq m\neq n \leq k\; :\; |x_m - x_n| < s/N\}$ are not too large, we do not need to ensure they are actually close to the expected value for a random sequence. Grepstad \& Larcher \cite{GL} discuss a notion similar to Poissonian pair discrepancy, but slightly different; we compare our notion to theirs in \S\ref{comparison_grepstad_larcher_section}. We can relate the correlation discrepancy to the sequence discrepancy as follows.

\begin{theorem}\label{discrepancy_seq_correlation}
For any sequence $(x_n)_{n=1}^{\infty}$, we have
\begin{equation}
    D_N \lesssim \frac{T}{N^k} + k\lp(\frac{k!}{T} + \frac{D^{(2k)}_{T,N}}{T}\rp)^{\frac{1}{2k}}.
\end{equation}
\end{theorem}
This explicit result shows that in the assumptions of Theorem \ref{kway}, we only need to assume Poissonian correlation (\ref{eq_2kfold_poissonian}) at integer values of $s$, not for all $s > 0$. We can also generalize Theorem \ref{discrepancy_seq_correlation} to deal with weak correlation. For any $0 \leq \alpha \leq k$, if 
\begin{equation}\label{assumption_general_alpha}
    \max_{0 < s < T} \frac{1}{N^{2k-\alpha}}\#\lp\{\substack{1 \leq a_j,b_j\leq N\\\{a_j\} \neq \{b_j\}} : \lp|\sum_{j=1}^k x_{a_j} - \sum_{j=1}^k x_{b_j}\rp| \leq \frac{s}{N^{\alpha}}\rp\} \leq 2s+F
\end{equation}
then
\begin{equation}\label{discrepancy_bound_seq_alpha}
    D_N \lesssim \frac{T}{N^{\alpha}} + k\lp(\frac{k!}{TN^{k-\alpha}} + \frac{F}{T}\rp)^{\frac{1}{2k}}.
\end{equation}
Consider the case $k = 1$, $\alpha = 0$; then $F$ is the discrepancy of the first $N^2 - N$ terms of the difference sequence $(x_m-x_n)_{1 \leq m\neq n \leq N}$, and we have 
\begin{equation*}
D_N \lesssim T + \frac{1}{\sqrt{T}}\lp(\frac{1}{N^2}+F\rp)^{1/2}.    
\end{equation*}
Because $F \geq 1/N^2$ this equation becomes $D_N \lesssim T + \sqrt{F/T}$. Optimizing for $T$ we obtain $D_N \lesssim F^{1/3}$, which is a quantitatively worse version of Theorem \ref{difference_discrepancy}. Indeed, this is Vinogradoff's 1926 bound \cite{V} for the same problem, which Cassels gave a simplified proof of \cite{CV}. Our proof seems to be different from both of theirs. 
\subsubsection{Comparison to other discrepancy bounds}
We can phrase Theorem \ref{discrepancy_seq_correlation} differently so as to more closely resemble similar inequalities in other papers. Suppose
\begin{equation*}
    \frac{1}{N^k}\#\lp\{\substack{1 \leq a_j,b_j \leq N\\\{a_j\} \neq \{b_j\}}\; :\; \lp|\sum_{j=1}^k x_{a_j} - \sum_{j=1}^k x_{b_j}\rp| < \frac{s}{N^k}\rp\} \leq (1+\delta)2s
\end{equation*}
for $s \in \{1,\ldots,T\}$. Then $\frac{D_{T,N}}{T} \leq \delta$, so by Theorem \ref{discrepancy_seq_correlation},
\begin{equation*}
    D_N \lesssim \frac{T}{N^k} + k^{3/2}T^{-\frac{1}{2k}} + \delta^{\frac{1}{2k}}
\end{equation*}
using $(a+b)^{\frac{1}{2k}} \leq a^{\frac{1}{2k}}+b^{\frac{1}{2k}}$ for $k \geq 1$. The case $k=1$ of pair correlation yields
\begin{equation}\label{discrepancy_delta_form_pair}
    D_N \lesssim \frac{T}{N} + \frac{1}{\sqrt{T}} + \sqrt{\delta}. 
\end{equation}
This form of the equation highlights that $\frac{D_{T,N}}{T}$ is a more fundamental quantity than $D_{T,N}$, and it is interesting to consider for fixed $\delta$ how our bound changes with $T$ and $N$. In the regime $\min\lp(\frac{1}{2}N^{2/5}, \frac{1}{\delta}\rp) \leq T \leq N^{2/5}$ we apply (\ref{discrepancy_delta_form_pair}) to obtain
\begin{equation*}
    D_N \lesssim N^{-1/5} + \sqrt{\delta}
\end{equation*}
which is Grepstad \& Larcher's discrepancy bound in  \cite{GL}. Their result has the advantage of being completely explicit in the constants. In the regime $T = 1/\delta$ we have
\begin{equation*}
    D_N \lesssim \frac{1}{N\delta} + \sqrt{\delta}
\end{equation*}
which refines Steinerberger's Theorem 1 in \cite{SSDisc}. We can use (\ref{discrepancy_delta_form_pair}) to answer Open Problem 1 from the same paper. Both Steinerberger's bound and Grepstad \& Larcher's bound only work for $T$ larger than some minimal value determined by $\delta$; motivated by this fact, Steinerberger asks for the minimal range $0 < s < T$ in which one needs to have control over the correlation discrepancy in order to control the sequence discrepancy. The answer is that there is no minimum range $T$; one immediately starts getting global regularity statements for any $T > 0$, all the way up to an optimal bound of $\sqrt{\delta}$ which is achieved for $T > 1/\delta$.
\subsubsection{Discrepancy bounds obtained for a random sequence}
As an example of the discrepancy bound in Theorem \ref{discrepancy_seq_correlation} we consider a sequence $(x_n)$ chosen uniformly and independently at random on $\TT$. Up to logarithmic factors we have $D^{(2k)}_{T,N} \approx T/N^k$, and this estimate yields 
\begin{equation}\label{random_seq_estimate_allk}
    D_N\lesssim k^{3/2}N^{-\frac{1}{2+1/k}}
\end{equation}
via Theorem \ref{discrepancy_seq_correlation}. Interestingly, this bound improves as $k$ increases. Up to logarithmic factors the discrepancy of a random sequence is $D_N \approx 1/\sqrt{N}$; so (\ref{random_seq_estimate_allk}) says that a sequence exhibiting Poissonian correlation of $2k$-fold differences for all $k$, and having correlation discrepancy comprable to a random sequence for all $k$, must have discrepancy close to that of a random sequence. 



\section{Proof of results}
\subsection{Proof of Theorem \ref{kway}}
Before proving that Poissonian correlation of $2k$-fold differences implies equidistribution, we state an equivalent property which we will use in our proof.
\begin{lemma}\label{lemma_integrate_test_function}
A sequence $(x_n)$ exhibits Poissonian correlation of $2k$-fold differences if and only if for all $f: \RR \to \RR$ piecewise $C^1$, compactly supported, and even we have
\begin{equation}\label{eq_testfunc_integrates}
    \lim_{N\to\infty} \frac{1}{N^{k}} \sum_{\substack{1\leq a_j,b_j\leq N\\\{a_1,\ldots,a_k\}\neq\{b_1,\ldots,b_k\}}} F_N\lp(\sum_{j=1}^k x_{a_j} - \sum_{j=1}^k x_{b_j}\rp) = \int_{-\infty}^{\infty} f(x)\ dx
\end{equation}
where
\begin{equation}\label{eq:FN_defn}
F_N(x) = \sum_{m\in\ZZ} f(N^k (x+m)).     
\end{equation} 
In addition, if $f$ is supported on $[-t,t]$, then we only need Poissonian correlation (\ref{eq_2kfold_poissonian}) to hold for all $0 < s < t$ in order for (\ref{eq_testfunc_integrates}) to hold.
\end{lemma}
The proof of this Lemma is standard, and in the case $k=1$ this is sometimes taken as the definition of Poissonian pair correlation. In particular, if $f = \chi_{[-s,s]}$, then Equation (\ref{eq_testfunc_integrates}) is the same as Equation (\ref{eq_2kfold_poissonian}) in the definition of $2k$-fold correlation. 
We now prove Theorem \ref{kway}.
\begin{proof}
Let $f$ be an arbitrary even, piecewise $C^1$ function compactly supported on $[-t,t]$ with $\int_{-\infty}^{\infty} f(x)dx = 1$. Let
\begin{equation}\label{equation_correlation_function}
\begin{aligned}
A^{(N)}(f) &= \frac{1}{N^k} \sum_{1\leq a_j,b_j\leq N} F_N\lp(\sum_{j=1}^k x_{a_j} - \sum_{j=1}^k x_{b_j}\rp), \\
R^{(N)}(f) &= \frac{1}{N^k} \sum_{\substack{1\leq a_j,b_j\leq N\\\{a_1,\ldots,a_k\}\neq \{b_1,\ldots,b_k\}}} F_N\lp(\sum_{j=1}^k x_{a_j} - \sum_{j=1}^k x_{b_j}\rp)
\end{aligned} 
\end{equation}
with $F_N$ as in (\ref{eq:FN_defn}).
We have
\begin{equation}\label{first_eq_FN}
    A^{(N)}(f) = \frac{Cf(0)}{N^k} + R^{(N)}(f)
\end{equation}
where $C = \#\{1\leq a_j,b_j\leq N : \{a_1,\ldots,a_k\} = \{b_1,\ldots,b_k\}\}$. Notice $C \leq k!N^k$. We can compute the quantity $A^{(N)}(f)$ in another way using Fourier series. We have
\begin{align*}
    \hat{F}_N(m) &= \int_{\TT} F_N(x) e^{-2\pi i m x}\ dx \\ 
    &= \int_{\TT} \sum_{c\in \ZZ} f(N^k(x+c))e^{-2\pi i m x}\ dx\\
    &= \frac{1}{N^k} \int_{\RR} f(y) e^{-2\pi i m y / N^k}\ dy \\
    &= \frac{1}{N^k} \hat{f}(m/N^k)
\end{align*}
so
\begin{align*}
    A^{(N)}(f) &= \frac{1}{N^k}\sum_{1 \leq a_j,b_j \leq N} \sum_{m\in \ZZ} \hat F_N(m) \exp \left(2\pi i m \left(\sum_{j=1}^k x_{a_j} - \sum_{j=1}^k x_{b_j}\right)\right) \\ 
    &= \frac{1}{N^{2k}} \sum_{m\in \ZZ} \hat{f}(m/N^k)\lp|\sum_{n=1}^N e^{2\pi i m x_n}\rp|^{2k} \\ 
    &= 1+\sum_{m\in \ZZ\setminus\{0\}}\hat{f}(m/N^k)\lp|\frac{1}{N}\sum_{n=1}^N e^{2\pi i m x_n}\rp|^{2k}.
\end{align*}
Substituting into Equation (\ref{first_eq_FN}) we get
\begin{equation}\label{key_ineq}
    2\sum_{m=1}^{\infty}\hat{f}(m/N^k)\lp|\frac{1}{N}\sum_{n=1}^N e^{2\pi i m x_n}\rp|^{2k} \leq k!f(0) + R^{(N)}(f) - 1
\end{equation}
using $C \leq k!N^k$. Using the characterization of Poissonian correlation of $2k$-fold differences from Lemma \ref{lemma_integrate_test_function}, we have $\lim_{N\to\infty} R^{(N)}(f) = 1$, so
\begin{equation}\label{function_inequality_k_corr}
    \limsup_{N\to\infty} \sum_{m=1}^{\infty}\hat{f}(m/N^k)\lp|\frac{1}{N}\sum_{n=1}^N e^{2\pi i m x_n}\rp|^{2k} \leq \frac{k!f(0)}{2}.
\end{equation}
Now we will choose $f$ carefully so we can obtain estimates on the Weyl sums from Lemma \ref{lem:weyl-sum-estimate}. Let $g(x) = \frac{1}{t}\chi_{[-t/2,t/2]}$, and $f = g * g$. Then 
\begin{equation*}
    \hat{g}(m) = \frac{\sin(\pi m t)}{\pi m t},\quad \hat{f}(m) = |\hat{g}(m)|^2 = \lp(\frac{\sin(\pi m t)}{\pi m t}\rp)^2
\end{equation*}
Explicitly, $f = \frac{1}{t^2}(t-|x|)\chi_{[-t,t]}$. Notice that $\hat{f}(m) \geq 0$, and $\hat{f}(m) \geq 1/2$ for $m < \frac{1}{4t}$. It follows that
\begin{equation*}
    \frac{1}{2}\sum_{m=1}^{N^k/(4t)} \lp|\frac{1}{N}\sum_{n=1}^N e^{2\pi i m x_n}\rp|^{2k} \leq \sum_{m=1}^{\infty} \hat{f}(m/N^k) \lp|\frac{1}{N}\sum_{n=1}^N e^{2\pi i m x_n}\rp|^{2k}
\end{equation*}
and we may apply inequality (\ref{function_inequality_k_corr}) to obtain Lemma \ref{lem:weyl-sum-estimate}
\begin{equation*}
    \limsup_{N\to\infty} \sum_{m=1}^{N^k/(4t)} \lp|\frac{1}{N}\sum_{n=1}^N e^{2\pi i m x_n}\rp|^{2k} \leq \frac{k!}{t}.
\end{equation*}
It follows that
\begin{equation*}
    \limsup_{N\to\infty}\lp|\frac{1}{N}\sum_{n=1}^N e^{2\pi i m x_n}\rp|^{2k} \leq \limsup_{N\to\infty} \sum_{m=1}^{N^k/(4t)} \lp|\frac{1}{N}\sum_{n=1}^N e^{2\pi i m x_n}\rp|^{2k} \leq \frac{k!}{t}\quad \text{for all } m > 0.
\end{equation*}
This equation holds for all $t > 0$, so we have 
\begin{equation*}
    \lim_{N\to\infty}\lp|\frac{1}{N}\sum_{n=1}^N e^{2\pi i m x_n}\rp|^{2k} = 0\quad \text{for all } m > 0,
\end{equation*}
thus
\begin{equation*}
    \lim_{N\to\infty}\lp|\frac{1}{N}\sum_{n=1}^N e^{2\pi i m x_n}\rp| = 0\quad \text{for all } m > 0
\end{equation*}
and by Weyl's theorem on equidistribution, $(x_n)$ is equidistributed on $\TT$. 
\end{proof}
Theorem \ref{poissonian_pair} is a special case of this result for $k=1$. In particular, our proof reduces to Steinerberger's proof of Theorem \ref{poissonian_pair} in \cite{SSHighDim} for the case $k=1$ . We did not use the full strength of Poissonian correlation in this proof; we only used the fact that for all $s > 0$,
\begin{equation*}
    \limsup_{N\to\infty}\frac{1}{N^k}\#\lp\{\substack{1 \leq a_j,b_j \leq N\\\{a_j\} \neq \{b_j\}}\; :\; \lp|\sum_{j=1}^k x_{a_j} - \sum_{j=1}^k x_{b_j}\rp| < \frac{s}{N^k}\rp\} \leq 2s.
\end{equation*}
This equation can be read as saying that not too many $2k$-fold differences are too close to 0. We show in \S\ref{section_discrepancy_bounds} that this equation only needs to hold for integer values of $s$, not all $s > 0$. Steinerberger derived Lemma \ref{lem:weyl-sum-estimate} in the case $k=1$, and argued that this estimate on the Weyl sums shows Poissonian pair correlation is a very strong property. In \S\ref{section_discrepancy_bounds} we bolster his claim by using Lemma \ref{lem:weyl-sum-estimate} to prove a discrepancy bound on the original sequence.

\subsection{Proof of Theorem \ref{kway_alpha}}
We now prove that Poissonian correlation of $2k$-fold differences at rate $\alpha$ implies equidistribution. We first modify Lemma \ref{lemma_integrate_test_function} for the case of weak correlation.
\begin{lemma}\label{lemma_integrate_test_function_weak}
A sequence $(x_n)$ exhibits weak Poissonian correlation of $2k$-fold differences at rate $\alpha$ if and only if for all $f: \RR \to \RR$ piecewise $C^1$, compactly supported, and even we have
\begin{equation}\label{eq_testfunc_integrates_weak}
    \lim_{N\to\infty} \frac{1}{N^{2k-\alpha}} \sum_{\substack{1\leq a_j,b_j\leq N\\\{a_1,\ldots,a_k\}\neq\{b_1,\ldots,b_k\}}} F_N\lp(\sum_{j=1}^k x_{a_j} - \sum_{j=1}^k x_{b_j}\rp) = \int_{-\infty}^{\infty} f(x)\ dx
\end{equation}
where
\begin{equation*}
F_N(x) = \sum_{m\in\ZZ} f(N^{\alpha} (x+m)).     
\end{equation*} 
In addition, if $f$ is supported on $[-t,t]$, then we only need weak correlation (\ref{eq_poissonian_rate_alpha_kfold}) to hold for all $0 < s < t$ in order for Equation (\ref{eq_testfunc_integrates_weak}) to hold.
\end{lemma}
We now prove Theorem \ref{kway_alpha}.
\begin{proof}
The setup is identical to the proof of Theorem \ref{kway}. Let $f$ be a piecewise $C^1$ function, even and compactly supported on $[-t,t]$. Analogously to (\ref{first_eq_FN}), we let
\begin{align*}
    A^{(N)}(f) &= \frac{1}{N^{2k-\alpha}} \sum_{1\leq a_j,b_j\leq N} F_N\lp(\sum_{j=1}^k x_{a_j} - \sum_{j=1}^k x_{b_j}\rp), \\
    R^{(N)}(f) &= \frac{1}{N^{2k-\alpha}} \sum_{\substack{1\leq a_j,b_j\leq N\\\{a_1,\ldots,a_k\}\neq \{b_1,\ldots,b_k\}}} F_N\lp(\sum_{j=1}^k x_{a_j} - \sum_{j=1}^k x_{b_j}\rp). 
\end{align*}
and obtain
\begin{equation*}
    A^{(N)}(f) = \frac{Cf(0)}{N^{2k-\alpha}} + R^{(N)}(f).
\end{equation*}
Next, we Fourier expand to obtain
\begin{equation*}
    A^{(N)}(f) = 1+\sum_{m\in \ZZ\setminus\{0\}}\hat{f}(m/N^\alpha)\lp|\frac{1}{N}\sum_{n=1}^N e^{2\pi i m x_n}\rp|^{2k}.
\end{equation*}
Combining these, we have
\begin{equation*}
    2\sum_{m=1}^{\infty}\hat{f}(m/N^\alpha)\lp|\frac{1}{N}\sum_{n=1}^N e^{2\pi i m x_n}\rp|^{2k} \leq \frac{k!f(0)}{N^{k-\alpha}} + R^{(N)}(f) - 1
\end{equation*}
where we use $C \leq k!N^k$. By Lemma \ref{lemma_integrate_test_function_weak}, $\lim_{N\to\infty} R^{(N)}(f) = 1$, so we can take a limit to obtain
\begin{equation*}
    \limsup_{N\to\infty}\sum_{m=1}^{\infty}\hat{f}(m/N^\alpha)\lp|\frac{1}{N}\sum_{n=1}^N e^{2\pi i m x_n}\rp|^{2k} \leq \limsup_{N\to \infty}\frac{k!f(0)}{2N^{k-\alpha}}.
\end{equation*}
For $f = \frac{1}{t^2}(t-|x|)\chi_{[-t,t]}$ as in the proof of Theorem \ref{kway}, we have
\begin{equation*}
    \limsup_{N\to\infty}\sum_{m=1}^{N^{\alpha}/(4t)}\lp|\frac{1}{N}\sum_{n=1}^N e^{2\pi i m x_n}\rp|^{2k} \leq \limsup_{N\to \infty}\frac{k!}{tN^{k-\alpha}}
\end{equation*}
and it follows that
\begin{equation}\label{ineq_with_alpha}
    \limsup_{N\to\infty} \lp|\frac{1}{N}\sum_{n=1}^N e^{2\pi i m x_n}\rp|^{2k} \leq \limsup_{N\to \infty}\frac{k!}{tN^{k-\alpha}}\quad \text{for all } m > 0.
\end{equation}
If $\alpha < k$, then we only need this equation to hold for some fixed $t > 0$ in order to conclude
\begin{equation*}
    \lim_{N\to\infty} \lp|\frac{1}{N}\sum_{n=1}^N e^{2\pi i m x_n}\rp|^{2k} = 0\quad \text{for all } m > 0,
\end{equation*}
which implies by Weyl's theorem that the sequence is equidistributed. If $\alpha = k$, then the right hand side of inequality (\ref{ineq_with_alpha}) is only small for large values of $t$. So for $\alpha < k$, in order to conclude equidistribution we need to have weak correlation (\ref{eq_poissonian_rate_alpha_kfold}) for all values $0 < s < t$, for some $t > 0$. For $\alpha = k$, we need to have correlation at all values $s > 0$ in (\ref{eq_poissonian_rate_alpha_kfold}).
\end{proof}

The proof of Theorem \ref{kway_alpha} illustrates that in order to conclude equidistribution, we need to control correlations at large enough distance scales. The right hand side of inequality (\ref{ineq_with_alpha}) is small if the scale we consider is large; for $\alpha = k$ the scale is large only when $t$ is large, but for $\alpha < k$ the scale is large even for $t$ small. Notice that for $\alpha = 0$, we have proved Weyl's result that if the sequence of differences $(x_m - x_n)$ is equidistributed then the original sequence $(x_n)$ is equidistributed. 

\section{Explicit Discrepancy Bounds}\label{section_discrepancy_bounds}
\subsection{Proof of discrepancy estimates}
We prove the discrepancy bound in Equation (\ref{discrepancy_estimate_exp_sum}).
\begin{proof}
By the Erd\"os-Tur\'an inequality (\ref{erdos-turan}), we have
\begin{equation*}
    D_N \lesssim \frac{1}{M}+\sum_{m=1}^M \frac{1}{m}\lp|\frac{1}{N}\sum_{n=1}^N e^{2\pi i m x_n}\rp|\quad \text{for all } M > 0.
\end{equation*}
Applying H\"older's inequality above with $p = 2k$, $q = \frac{2k}{2k-1}$ we obtain
\begin{equation*}
    \sum_{m=1}^M \frac{1}{m}\lp|\frac{1}{N}\sum_{n=1}^N e^{2\pi i m x_n}\rp| \leq \lp(\sum_{m=1}^M\lp|\frac{1}{N}\sum_{n=1}^N e^{2\pi i m x_n}\rp|^{2k}\rp)^{\frac{1}{2k}}\lp(\sum_{m=1}^M m^{-\frac{2k}{2k-1}}\rp)^{\frac{2k-1}{2k}}
\end{equation*}
Using an integral estimate, 
\begin{equation*}
    \sum_{m=1}^M m^{-\frac{2k}{2k-1}} \leq 1 + \int_1^{\infty} x^{-\frac{2k}{2k-1}}\ dx = 2k
\end{equation*}
So 
\begin{equation*}
    D_N \lesssim \frac{1}{M}+k\lp(\sum_{m=1}^M\lp|\frac{1}{N}\sum_{n=1}^N e^{2\pi i m x_n}\rp|^{2k}\rp)^{\frac{1}{2k}}\quad \text{for all } M > 0.
\end{equation*}
Letting $M = N^k/(4T)$ we obtain Equation (\ref{discrepancy_estimate_exp_sum}).
\end{proof}
We now prove Theorem \ref{discrepancy_seq_correlation}, providing a bound on the discrepancy of a sequence in terms of the correlation discrepancy. 
\begin{proof}
First we prove a quantitative version of Lemma \ref{lemma_integrate_test_function}, obtaining a quantative bound on the correlation of a test function. Let $f(x) = \frac{1}{T^2}\lp(T - |x|\rp)\chi_{[-T,T]}$ as in the proof of Theorem \ref{kway}, where $T > 0$ is an integer. Then we show
\begin{equation}\label{inequality_integrate_func}
    R^{(N)}(f) \leq 1 + \frac{1}{T} + \frac{D_{T,N}}{T}.
\end{equation}
with $R^{(N)}(f)$ as in (\ref{equation_correlation_function}). Let $g(x)$ be the step function
\begin{equation*}
    g(x) = \frac{1}{T^2}\lp(\chi_{[-T,T]} + \chi_{[-(T-1),T-1]} + \cdots + \chi_{[-1,1]}\rp),
\end{equation*}
then 
\begin{equation*}
    R^{(N)}(g) = \frac{1}{T^2}\sum_{s=1}^T P_{s,N}
\end{equation*}
where 
\begin{equation*}
    P_{s,N} = \frac{1}{N^k}\#\lp\{\substack{1 \leq a_j,b_j\leq N\\\{a_j\}\neq \{b_j\}}\; :\; \lp|\sum_j x_{a_j} - \sum_j x_{b_j}\rp| < \frac{s}{N^k}\rp\},
\end{equation*}
and we have $P_{s,N} \leq 2s + D_{T,N}$. We have $\frac{1}{T^2}\sum_{s=1}^T2s = \int_{-\infty}^{\infty} g(x)dx = 1 + \frac{1}{T}$, so
\begin{equation*}
    R^{(N)}(g) \leq 1 + \frac{1}{T} + \frac{D_{T,N}}{T}.
\end{equation*}
Because $f \leq g$, we also have $F_N \leq G_N$, and we obtain (\ref{inequality_integrate_func}) from the above. Substituting this bound in (\ref{key_ineq}) we obtain
\begin{equation*}
    \sum_{m=1}^{N^k/(4t)}\lp|\frac{1}{N}\sum_{n=1}^N e^{2\pi i m x_n}\rp|^{2k} \leq \frac{k!}{T} + \frac{1}{T} + \frac{D_{T,N}}{T}
\end{equation*}
plugging this inequality into Equation (\ref{discrepancy_estimate_exp_sum}), we obtain
\begin{equation*}
    D_N \lesssim \frac{T}{N^k} + k\lp(\frac{k!}{T} + \frac{D_{T,N}}{T}\rp)^{\frac{1}{2k}}
\end{equation*}
as desired.
\end{proof}
Equation (\ref{discrepancy_bound_seq_alpha}) generalizes this bound to all $0 \leq \alpha \leq k$; here we assume the correlation discrepancy bound holds for all $0 < s < T$, rather than just integer values of $s$. 
\begin{proof}
We prove that under the assumption (\ref{assumption_general_alpha}), $R^{(N)}(f) \leq 1 + \frac{F}{T}$. For any $L > 0$ an integer, let $c = T/L$. Let $g(x)$ be the step function
\begin{equation*}
    g(x) = \frac{c}{T^2}\lp(\chi_{[-c,c]} + \chi_{[-2c,2c]} + \cdots + \chi_{[-Lc,Lc]}\rp)
\end{equation*}
Then analogously to before, $R^{(N)}(g) \leq 1 + \frac{c}{T} + \frac{F}{T}$. This equation holds for any $L > 0$, so $R^{(N)}(g) \leq 1 + \frac{F}{T}$ and so it follows that
\begin{equation*}
    \sum_{m=1}^{N^{\alpha}/(4t)}\lp|\frac{1}{N}\sum_{n=1}^N e^{2\pi i m x_n}\rp| \leq \frac{k!}{TN^{k-\alpha}} + \frac{F}{T}.
\end{equation*}
Applying the Erd\"os Turan inequality with $M = N^{\alpha}/(4T)$, we obtain
\begin{equation*}
    D_N \lesssim \frac{T}{N^{\alpha}} + k\lp(\frac{k!}{TN^{k-\alpha}} + \frac{F}{T}\rp)^{\frac{1}{2k}}
\end{equation*}
as desired.
\end{proof}

\subsection{Comparison with Grepstad and Larcher's discrepancy bound}\label{comparison_grepstad_larcher_section}
For $(x_1,\ldots,x_N)$ a collection of points on $\TT$, let
\begin{equation*}
    F_{T,N} = \max_{s=1,\ldots,T} \lp|\frac{1}{2s}\#\lp\{1 \leq m \neq n \leq N\ :\ |x_m - x_n| \leq \frac{s}{N}\rp\} - N\rp|.
\end{equation*}
Grepstad \& Larcher \cite{GL} prove the discrepancy bound 
\begin{equation*}
    D_N \lesssim N^{-1/5}+ \sqrt{\frac{F_{T^2,N}}{N}}
\end{equation*}
in the regime
\begin{equation*}
    \min\lp(\frac{1}{2}N^{2/5},\frac{N}{F_{T^2,N}}\rp) \leq T \leq N^{2/5}.
\end{equation*}
We have $F_{T,N} \geq \frac{ND_{T,N}}{T}$, so our bound from Theorem \ref{discrepancy_seq_correlation} yields in the case of Pair correlation 
\begin{equation*}
    D_N \lesssim \frac{T}{N} + \frac{1}{\sqrt{T}} + \sqrt{\frac{F_{T,N}}{N}}
\end{equation*}
In the regime Grepstad \& Larcher discuss, $\frac{T}{N} \leq N^{-1/5}$ and $\frac{1}{\sqrt{T}} \lesssim \max\lp(N^{-1/5}, \sqrt{\frac{F_{T,N}}{N}}\rp)$. It follows from Theorem \ref{discrepancy_seq_correlation} that $D_N \lesssim N^{-1/5} + \sqrt{\frac{F_{T,N}}{N}}$ which is Grepstad \& Larcher's bound.

\subsection{Comparison with Steinerberger's discrepancy bound} 
In Theorem 1 of \cite{SSDisc}, Steinerberger proves that if for some $0 < \delta < 1/2$ a set of points $\{x_1,\ldots,x_N\}$ satisfies
\begin{equation*}
    \frac{1}{N}\#\lp\{1 \leq m \neq n \leq N : |x_m-x_n| \leq \frac{s}{N}\rp\} \leq (1+\delta)2s\quad \text{ for all } 1 \leq s \leq (8/\delta)\sqrt{\log N},
\end{equation*}
then $D_N \lesssim \delta^{1/3} + N^{-1/3}\delta^{-1/2}$. In this situation we have by (\ref{discrepancy_delta_form_pair})
\begin{equation*}
    D_N\lesssim \frac{T}{N} + \frac{1}{\sqrt{T}} + \sqrt{\delta},
\end{equation*}
and letting $T = \frac{1}{\delta} < (8/\delta)\sqrt{\log N}$ we obtain $D_N \lesssim \frac{1}{N\delta} + \sqrt{\delta}$. For $N > \delta^{-3/4}$, $\frac{1}{N\delta}\leq N^{-1/3}\delta^{-1/2}$ so our bound is stronger. For $N < \delta^{-3/2}$, $N^{-1/3}\delta^{-1/2} > 1$ and Steinerberger's bound is trivial. Because $\delta^{-3/2} > \delta^{-3/4}$, we have 
\begin{equation*}
    \frac{1}{N\delta} + \sqrt{\delta} \lesssim \delta^{1/3} + N^{-1/3}\delta^{-1/2}
\end{equation*}
and (\ref{discrepancy_delta_form_pair}) refines Steinerberger's bound.

\subsection{Explicit discrepancy bounds for a random sequence}
We give a proof sketch for Equation (\ref{random_seq_estimate_allk}), which is the optimal discrepancy bound one can obtain from Theorem \ref{discrepancy_seq_correlation} for a sequence having $2k$-fold correlation discrepancy similar to that of a random sequence. 
\begin{proof}
Let $(x_n)$ be a sequence of points chosen uniformly and indepenently at random on $\TT$. Then up to logarithmic a factor,
\begin{equation*}
    D_{T,N}^{(2k)} \lesssim \sigma(X)
\end{equation*}
where $\sigma(X)$ denotes the standard deviation, and $X$ is a binomial random variable with a probability $p = \frac{2T}{N^k}$ of taking the value $\frac{1}{N^k}$, $1-p$ of value 0, and $N^{2k}$ samples. We have 
\begin{equation*}
    \sigma^2(X) = N^{2k}\frac{1}{N^{2k}}p(1-p) \approx \frac{2T}{N^k}
\end{equation*}
So up to logarithmic factors, $D_{T,N}^{(2k)} \lesssim \sqrt{T/N^k}$. Let $T = N^{\beta}$, with $\beta$ to be chosen later. Then we have
\begin{equation*}
    D_N \lesssim N^{\beta - k} + k\lp(k!N^{-\beta} + N^{-k/2}T^{-\beta/2}\rp)^{\frac{1}{2k}} \lesssim N^{\beta-k} + k^{3/2}N^{-\frac{\beta}{2k}} + kN^{-\frac{1+\beta/k}{4}}
\end{equation*}
where we use $(k!)^{\frac{1}{2k}} \lesssim k^{1/2}$. 
The minimum of these three terms is largest asymptotically when $\beta - k = -\frac{\beta}{2k}$, or when $\beta = \frac{k}{1+1/(2k)}$. In this case $\beta - k = -\frac{1}{2+1/k}$, so we obtain $D_N\lesssim k^{3/2}N^{-\frac{1}{2+1/k}}$ up to logarithmic factors, as desired.    
\end{proof}

\section{Examples}\label{section-examples}
We exhibit some examples differentiating Poissonian pair correlation from higher order correlation. The two concepts are incomparable: the property of being Poissonian pair correlated does not imply four-fold correlation, and having four-fold correlation does not imply Poissonian pair correlation. 

\begin{example}\label{pair_not_fourfold}
Let 
\begin{equation*}
(y_n) = (x_1, -x_1, x_2, -x_2, \ldots)    
\end{equation*}
where $x_n$ is chosen uniformly and independently at random on $\TT$. Then $(y_n)$ almost surely exhibits Poissonian pair correlation, but does not exhibit Poissonian correlation of four-fold differences. This example illustrates how correlation of $2k$-fold differences is a global notion, whereas Poissonian pair correlation is local

\begin{proof}
First we show $(y_n)$ does not have Poissonian correlation of four-fold differences. Indeed, we have
\begin{equation*}
    \#\lp\{\substack{1 \leq m,n,k,l \leq N\\\{m,n\} \neq \{k,l\}}\; :\; \lp|y_m + y_n - y_k - y_l\rp| = 0\rp\} \geq N^2/4-N/2
\end{equation*}
because for any we have $1 \leq a\neq b\leq N/2$,
\begin{equation*}
    y_{2a} + y_{2a+1} - y_{2b} - y_{2b+1} = x_a - x_a + x_b - x_b = 0
\end{equation*}
and there are $N^2/4 - N/2$ such choices of $a$, $b$. Notice that the obstruction to $(y_n)$ being four-fold correlated is global; in general, the four points $x_a, -x_a, x_b, -x_b$ will be far away from each other on $\TT$. On the other hand, $(y_n)$ is Poissonian pair correlated. For $N$ even, we have
\begin{equation*}
    \#\lp\{1\ \leq m \neq n \leq N : |y_m - y_n| \leq \frac{s}{N}\rp\} = 2A_N + 2B_N
\end{equation*}
where 
\begin{equation*}
    A_N = \#\left\{1 \leq a \neq b \leq \frac{N}{2} : |x_a - x_b| \leq \frac{s}{N}\right\},\quad B_N = \#\left\{1 \leq a \neq b \leq \frac{N}{2}: |x_a + x_b| \leq \frac{s}{N}\right\}.
\end{equation*}
Just as a uniformly random and i.i.d sequence exhibits Poissonian pair correlation, we have
\begin{equation*}
    \lim_{N\to\infty} \frac{A_N}{N}  = s\ \text{ and }\ \lim_{N\to\infty} \frac{B_{N}}{N} = s\quad \text{a.s}
\end{equation*}
so it follows almost surely that
\begin{equation*}
    \lim_{N\to\infty} \frac{1}{N}\#\lp\{1\ \leq m \neq n \leq N\; :\; |y_m - y_n| \leq \frac{s}{N}\rp\} = \lim_{N\to\infty} \frac{A_N+B_N}{N} = 2s
\end{equation*}
and $(y_n)$ is almost surely Poissonian pair correlated.
\end{proof}
\end{example}

\begin{example}\label{additive-energy-example}
We exhibit another sequence having Poissonian pair correlation but not Poissonian correlation of four-fold differences. In \cite{ACL}, Aistleitner, Larcher, \& Lewko prove the following theorem relating Poissonian pair correlation to the additive energy of sequences.
\begin{theorem}[Aistleitner, Larcher, \& Lewko, 2017]\label{poissonian_pair_acl}
If $(a_n)$ is a strictly increasing sequence of integers such that
\begin{equation*}
    E(a_1,\ldots, a_N) = \mathcal{O}(N^{3-\varepsilon})
\end{equation*}
for some $\varepsilon > 0$, then $(\{a_n\alpha\})_{n\geq 1}$ has Poissonian pair correlation for almost all $\alpha$. 
\end{theorem}
Here,
\begin{equation*}
    E(a_1,\ldots, a_N) := \#\{1 \leq m,n,k,l \leq N\; :\; a_m - a_n = a_k - a_l\}
\end{equation*}
is the additive energy, and $\{a_n\alpha\}$ is the fractional part. Let $(a_n)$ be a sequence with additive energy satisfying $E(a_1,\ldots, a_N) = \mathcal{O}(N^{3-\varepsilon})$ and $E(a_1,\ldots,a_N) \geq CN^2$ for some $C > 1$, all large enough $N$. For $N$ large enough, 
\begin{equation*}
    \frac{1}{N^2}\#\lp\{\substack{1 \leq m,n,k,l \leq N\\\{m,n\}\neq\{k,l\}}\; :\; |a_m\alpha + a_n\alpha - a_k\alpha - a_l\alpha| = 0\rp\} \geq C-1
\end{equation*}
so by considering $s < C - 1$ in Equation (\ref{eq_2kfold_poissonian}), we see the sequence $(\{a_n\alpha\})_{n=1}^{\infty}$ cannot exhibit Poissonian correlation of four-fold differences. Yet for almost all $\alpha$, $(\{a_n\alpha\})_{n=1}^{\infty}$ has Poissonian pair correlation, so for some irrational $\alpha$, $(\{a_n\alpha\})_{n=1}^{\infty}$ has Poissonian pair correlation but not four-fold Poissonian correlation. Notice that in this example, Equation (\ref{eq_2kfold_poissonian}) may still hold for $s$ large; yet if we enforce the stronger condition $\frac{1}{N^2}E(a_1, \ldots, a_N) \to \infty$, then
\begin{equation*}
    \frac{1}{N^2}\#\lp\{\substack{1 \leq m,n,k,l \leq N\\\{m,n\}\neq\{k,l\}}\; :\; |a_m\alpha + a_n\alpha - a_k\alpha - a_l\alpha| = 0\rp\} \to \infty
\end{equation*}
as well, and the sequence $(\{a_n\alpha\})_{n=1}^{\infty}$ will not exhibit four-fold correlation for \textit{any} value of $s$, and any $\alpha$. 
\end{example}

\begin{example}\label{fourfold_not_pair}
We present a random sequence which almost surely exhibits Poissonian correlation of four-fold differences, but not Poissonian pair correlation. Let $(y_n)$ be the sequence given by
\begin{equation*}
    y_{n} = x_n + p_n
\end{equation*}
where $x_1, x_3, \ldots$ are chosen uniformly and independently at random on $\TT$, $x_{2n-1} = x_{2n}$ for all $k \geq 1$, and $p_n$ is chosen uniformly and independently at random on the inverval $[n^{-3/2}, n^{3/2}]$. This sequence is chosen so that $\frac{x_{2n-1} + x_{2n}}{2}$ is uniformly random in $\TT$, but $\frac{x_{2n-1} - x_{2n}}{2}$ lies at scale $n^{-3/2}$. 
We show that $(y_n)$ almost surely does not have Poissonian pair correlation, but does have Poissonian correlation of four-fold differences.
\begin{proof}
First, $(y_n)$ satisfies
\begin{align*}
    \#\left\{1 \leq m \neq n \leq N : |y_m - y_n| \leq \frac{s}{N}\right\} &\geq \# \left\{1 \leq n \leq N/2 : |p_{2n-1} - p_{2n}|  \leq \frac{s}{N}\right\}
\end{align*}
and
\begin{align*}
    \limsup_{N\to \infty} \frac{1}{N}\# \left\{1 \leq n \leq N/2 : |p_{2n-1} - p_{2n}|  \leq \frac{s}{N}\right\} \geq \frac{1}{8} \quad \text{almost surely}
\end{align*}
because $|p_n| \leq n^{-3/2}$. Taking $s$ sufficiently small in (\ref{eq:pair_correlation}), we see $(y_n)$ almost surely does not have Pair correlation. On the other hand, $(y_n)$ almost surely does exhibit Poissonian correlation of four-fold differences. The idea is that the scale of Poissonian pair correlation is $s/N$ whereas the scale of four-fold correlation is $s/N^2$; adding perturbations of scale $N^{-3/2}$ destroys Poissonian pair correlation but maintains four-fold correlation. Let
\begin{align*}
    A_N &= \#\Big\{\substack{1\leq m,n,k,l\leq N\\\{m,n\}\neq\{k,l\}} : |y_m+y_n-y_k-y_l| \leq \frac{s}{N^2}, x_m+x_n-x_k-x_l \neq 0\Big\}, \\ 
    B_N &= \#\Big\{\substack{1\leq m,n,k,l\leq N\\\{m,n\}\neq\{k,l\}} : |y_m+y_n-y_k-y_l| \leq \frac{s}{N^2}, x_m+x_n-x_k-x_l = 0\Big\},
\end{align*}
where we suppress the implicit dependence of $A_N$, $B_N$ on $s$. 
We would like to show that
\begin{equation}\label{eq:AB_poissonian}
    \lim_{N\to \infty} \frac{A_N+B_N}{N^2} = 2s \quad \text{almost surely}
\end{equation}
for all $s > 0$. 
\begin{claim*}
For all $s > 0$, we have
\begin{align*}
    \lim_{N\to \infty} \frac{A_N}{N^2} &= 2s \quad \text{almost surely}, \\ 
    \lim_{N\to \infty} \frac{B_N}{N^2} &= 0 \quad \text{almost surely}.
\end{align*}
\end{claim*}
\begin{proof}
The first line follows from the same argument that a uniformly random and i.i.d sequence almost surely exhibits Poissonian pair correlation. For the second statement, notice that if $x_m + x_n - x_k - x_l = 0$, then 
\begin{equation*}
    y_m + y_n - y_k - y_l = p_m + p_n - p_k - p_l
\end{equation*}
and for $m,n,k,l < N$, 
\begin{align*}
    \PP\left(|p_m + p_n - p_k - p_l| < \frac{s}{N^2}\right) &\leq \PP\left(N^{-3/2}|p_m m^{3/2} + p_m n^{3/2} - p_k k^{3/2} - p_l l^{3/2}| < sN^{-2}\right) \\ 
    &\leq \PP\left(|p_m m^{3/2} + p_m n^{3/2} - p_k k^{3/2} - p_l l^{3/2}| < sN^{-1/2}\right) \\ 
    &\leq 2sN^{-1/2}
\end{align*}
as $p_m m^{3/2}$ and other such variables are independently and uniformly distributed in $\TT$. Now by linearity of expectation,
\begin{equation*}
    \EE[B_N] \leq 2sN^{-1/2} \#\Big\{\substack{1\leq m,n,k,l\leq N\\\{m,n\}\neq\{k,l\}} : x_m+x_n-x_k-x_l = 0\Big\}.
\end{equation*}
Almost surely, we can only have $x_m+x_n-x_k-x_l = 0$ if $k, l \in \{m \pm 1, n \pm 1\}$, and there are at most $16N^2$ such choices of $m,n,k,l$. So $\EE[B_N] \leq 32sN^{3/2}$, and by Markov's inequality, it follows that 
\begin{equation*}
    \lim_{N\to \infty} \frac{B_N}{N^2} = 0 \quad \text{almost surely}
\end{equation*}
as desired.
\end{proof}

Now that the Claim has been proven, Equation (\ref{eq:AB_poissonian}) follows immediately, and we see that $(y_n)$ almost surely exhibits Poissonian correlation of four-fold differences. 
\end{proof}
\end{example}

\begin{example}\label{fourfold_not_pair_rate1}
The sequence in Example \ref{fourfold_not_pair} exploited the difference in scaling between Poissonian pair correlation and four-fold correlation; pair correlation takes place at scale $1/N$, and four-fold correlation takes place at scale $1/N^2$. We may instead compare pair correlation to four-fold correlation at rate $\alpha = 1$, as here the two notions take place at the same scale. Recall that a sequence has four-fold correlation at rate $\alpha = 1$ if for all $s > 0$,
\begin{equation*}
    \lim_{N\to\infty} \frac{1}{N^3} \#\lp\{\substack{1 \leq m,n,k,l \leq N\\\{m,n\} \neq \{k,l\}}\; :\; |x_m+x_n-x_k-x_l| \leq \frac{s}{N}\rp\} = 2s.
\end{equation*}
Let $(y_n)$ be the sequence given by
\begin{equation*}
    (y_n) = (x_1, x_1, x_2, x_2, \ldots).
\end{equation*}
where $x_n$ is chosen uniformly and independently at random on $\TT$. Then $(y_n)$ does not have Poissonian pair correlation, nor does it have four-fold Poissonian correlation at rate $\alpha = 2$, but it does have four-fold Poissonian correlation at rate $\alpha = 1$. 
\begin{proof}
First,
\begin{equation*}
    \frac{1}{N} \#\{1 \leq m \neq n \leq N\; :\; |y_m - y_n| = 0\} \geq \frac{1}{2}
\end{equation*}
so $(y_n)$ is not Poissonian pair correlated. Also, $(y_n)$ does not have Poissonian correlation of four-fold differences at rate $\alpha = 2$ because
\begin{equation*}
    \frac{1}{N^2} \#\lp\{\substack{1 \leq m,n,k,l \leq N\\\{m,n\}\neq \{k,l\}}\; :\; |y_m+y_n-y_k-y_l| = 0\rp\} \geq \frac{1}{2}.
\end{equation*} 
However, $(y_n)$ does have Poissonian correlation of four-fold differences at rate $\alpha = 1$. Notice that each random variable $y_m+y_n-y_k-y_l$ is either uniformly distributed or equal to 0. Let 
\begin{align*}
    A_N &= \#\lp\{\substack{1\leq m,n,k,l\leq N\\\{m,n\}\neq\{k,l\}}\; :\; 0 < |y_m+y_n-y_k-y_l| < \frac{s}{N}\rp\}, \\
    B_N &= \#\lp\{\substack{1\leq m,n,k,l\leq N\\\{m,n\}\neq\{k,l\}}\; :\; |y_m+y_n-y_k-y_l| = 0\rp\}.
\end{align*}
As we demonstrated in Example \ref{fourfold_not_pair}, $B_N \leq 16N^2$ almost surely, so 
\begin{equation*}
    \lim_{N\to \infty} \frac{B_N}{N^3} = 0 \quad \text{almost surely.}
\end{equation*}
Also, the proof that a sequence chosen uniformly and independently at random is Poissonian pair correlated extends to show
\begin{equation*}
    \lim_{N\to\infty} \frac{A_N}{N^3} = 2s\quad \text{almost surely.}
\end{equation*}
So we have 
\begin{equation*}
    \lim_{N\to\infty} \frac{1}{N^3} \#\lp\{\substack{1 \leq m,n,k,l \leq N\\\{m,n\}\neq\{k,l\}}\; :\; |x_m+x_n-x_k-x_l| \leq \frac{s}{N}\rp\} = \lim_{N\to\infty} \frac{A_N}{N^3} + \frac{B_N}{N^3} = 2s
\end{equation*}
almost surely, and we are done.
\end{proof}
\end{example}

\section{Conclusion}
A main difficulty in studying Poissonian pair correlation is a lack of examples. Theorem \ref{kway} provides a countable collection of conditions each of which are stronger than equidistribution and similar to Poissonian pair correlation. It is plausible that it will be easier to find a sequence that exhibits Poissonian correlation of $2k$-fold differences for $k>1$ than it is to find a sequence exhibiting Poissonian pair correlation. The motivation is that in taking higher order differences we destroy more information from the original sequence, and are left with a sequence that is more random. 
\begin{question}
Do any explicit and common sequences exhibit Poissonian correlation of four-fold differences, but not Poissonian pair correlation?
\end{question}

We may also ask a more quantitative version of the same idea. We say that an increasing sequence $(a_n)$ of positive integers has the \textit{metric pair correlation property} if for almost all $\alpha$, $(\{a_n\alpha\})$ has Poissonian pair correlation. In a similar way, we say that $(a_n)$ has the \textit{metric $2k$-fold correlation property} if for almost all $\alpha$, $(\{a_n\alpha\})$ has Poissonian correlation of $2k$-fold differences. Quantitatively, one can compare the size of the exceptional sets by their Hausdorff dimension. The Hausdorff dimension of the set of $\alpha$ for which $(\{a_n\alpha\})$ is not Poissonian pair correlated has been intensely studied, for instance in \cite{ACL}. It would be interesting to compare the Hausdorff dimension of the exceptional set for the Poissonian pair correlation property and the $2k$-fold correlation property. In particular, one might expect that because taking higher order correlations destroys more information, the Hausdorff dimension of the exceptional set decreases for larger values of $k$.
\begin{question}
Let $(a_n)$ a sequence of integers having the metric $2k$-fold correlation property and the metric $2m$-fold correlation property, with $k < m$. Let 
\begin{equation*}
    A = \{\alpha : (\{a_n\alpha\}) \text{ not $2k$-fold correlated}\}, B = \{\alpha : (\{a_n\alpha\}) \text{ not $2m$-fold correlated}\}.
\end{equation*}
Will the Hausdorff dimension of $A$ be greater than that of $B$?
\end{question}
Example \ref{additive-energy-example} shows that in order for this question to be interesting, the sequence $(a_n)$ cannot have additive energy (or higher order analogues) too large--otherwise, it won't exhibit Poissonian correlation of $2k$-fold differences for any $\alpha$. 
A positive answer to this question would strengthen our notion that Poissonian correlation of higher order differences is a weaker property than Poissonian correlation of lower order differences. 
Other approaches to comparing the relative strength of Poissonian pair correlation and Poissonian correlation of higher order differences would be interesting routes for further study. 

\section{Acknowledgements}
Many thanks to Stefan Steinerberger for proposing the question and for valuable discussions during the writing of this paper.\\\\
Declerations of interest: none

\end{document}